\documentclass[11pt,a4paper,twoside]{article}
\usepackage{color}
\usepackage{amsmath, amssymb, amsthm, graphicx} 

\def\RR{\mathbb{R}}
\newcommand\tr{\operatorname{trace}}
\def\Ric{\operatorname{Ric}}

\newtheorem{corollary}{Corollary}
\newtheorem{definition}{Definition}
\newtheorem{example}{Example}
\newtheorem{remark}{Remark}
\newtheorem{lemma}{Lemma}
\newtheorem{proposition}{Proposition}
\newtheorem{theorem}{Theorem}

\author{Mirjana Djori\'c \footnote{
University of Belgrade, Faculty of Mathematics, Studentski trg 16, 11000 Belgrade, Serbia
       \newline e-mail: {\tt mdjoric@matf.bg.ac.rs}}
       \ and\
       Vladimir Rovenski \footnote{University of Haifa, Department of Mathematics, Mount Carmel, 3498838, Haifa, Israel
       \newline e-mail: {\tt vrovenski@univ.haifa.ac.il}
       }
}

\title{Geometric inequalities for CR-submanifolds}

\begin{document}

\date{}

\maketitle

\begin{abstract}
We study two kinds of curvature invariants of Riemannian manifold equip\-ped with a complex distribution $D$
(for example, a CR-submanifold of an almost Hermitian manifold) related to sets of pairwise orthogonal subspaces of the distribution.
One kind of invariant is based on the mutual curvature of the subspaces and another is similar to Chen's $\delta$-invariants.
We compare the mutual curvature invariants with Chen-type invariants
and prove geometric inequalities with intermediate mean curvature squared for CR-submanifolds in almost Hermitian spaces.
In the case of a set of complex planes, we introduce and study
curvature invariants based on the concept of holomorphic bisectional curvature.
As~applications, we give consequences of the absence of some $D$-minimal CR-submanifolds in almost Hermitian~manifolds.


\end{abstract}

\section{Introduction}

In 1978, A. Bejancu introduced the notion of a CR-submanifold as a generalization of holomorphic and totally real submanifolds of almost Hermitian manifolds.
Since then, CR-submanifolds in various ambient spaces have been actively studied, for example,~\cite{CSV-24,DO-2010,YK-1983}.

The development of the extrinsic geometry of submanifolds led to the following problem
(for example,~\cite{chen1}): \textit{find a simple optimal connection between the intrinsic and extrinsic invariants of a submanifold
in a Riemannian manifold; in particular, in space~forms}.

B.Y.~Chen introduced the concept of $\delta$-curvature invariants for Riemannian manifolds in the 1990s
and proved an optimal inequality for a subma\-nifold, including $\delta$-curvature invariants and the square of mean curvature.
Chen invariants are obtained from the scalar curvature by discarding some of the sectional~curvatures.
The case of equality led to the notion of ``ideal immersions'' in Euclidean space, that is, isometric immersions with the smallest possible tension.
Chen's theory was extended by geometers to K\"{a}hler, (para-)contact, Lagrangian and affine submanifolds, warped products and submersions, see~\cite{chen1,CSV-24}.

\smallskip

In \cite{r-w-2022,rov-112}, we introduced invariants of a Riemannian manifold,
which are related to the mutual curvature of noncomplementary pairwise orthogonal subspaces of the tangent bundle,
and proved
geometrical
inequalities for Riemannian submanifolds with applications to foliations.

In the paper, we study two kinds of curvature invariants of a Riemannian manifold equipped with a complex distribution $D$
(in particular, a CR-submanifold of an almost Hermitian mani\-fold), related to sets of pairwise orthogonal subspaces of the distribution.
One kind of invariant is based on the mutual curvature of the subspaces and the other is similar to Chen's $\delta$-invariants.
We compare the mutual curvature invariants with Chen-type invariants
and prove geometric inequalities with intermediate mean curvature squared for CR-submanifolds in almost Hermitian spaces.
In the case of a set of complex planes, we introduce and study curvature invariants based on the holomorphic bisectional curvature of two planes.
As~applications, we provide the consequences of the absence of some ${D}$-minimal CR-submanifolds.

The paper is organized as follows.
In Section~\ref{sec:02}, we report some basic information about the curvature invariants of a manifold with a distribution.
In Section~\ref{sec:03}, we study geometric inequalities for CR-submanifolds in almost Hermitian manifolds.

\section{Curvature invariants of a manifold with a distribution}
\label{sec:02}

In this section, we recall
two kinds of curvature invariants of a manifold with a distribution (Chen-type invariants and invariants based on the mutual curvature, see \cite{r-w-2022,rov-112}), and
for the complex distribution we define invariants based on the holomorphic bisectional curvature.
Let $(M^{d+l},g)\ (d,m>0)$ be a Riemannian manifold with a $d$-dimensional distribution $D$.
Denote by $\nabla$ the Levi-Civita connection of $g$ and
$R_{X,Y} = \nabla_{X}\nabla_{Y} -\nabla_{Y}\nabla_{X} -\nabla_{[X,\, Y]}$ the curvature tensor,
where $X,Y$ are any vector fields on the tangent bundle $TM$.
The~scalar curvature $\tau$ (function on $M$) is the trace of the Ricci tensor $\Ric_{X,Y}=\tr( Z\mapsto R_{Z,X}\,Y)$.
Some authors, for example, \cite{chen1}, define the scalar curvature as half of ``trace~Ricci''.

\begin{example}\rm
Let $g$ be an admissible metric for an almost contact manifold $M^{2n+1}(\varphi,\xi,\eta)$,
\begin{equation*}
 g(\varphi X,\varphi Y)= g(X,Y) -\eta(X)\,\eta(Y),\quad
 \eta(\xi)=1,\quad X,Y\in\mathfrak{X}_M,
\end{equation*}
see \cite{blair2010riemannian},
where $\varphi$ is a $(1,1)$-tensor, $\xi$ is a unit vector field (called Reeb vector field) and $\eta$ is a
1-form. Then $d=2n$ and  $D=\ker\eta$ is a $2n$-dimensional contact distribution on $M^{2n+1}$.
\end{example}

Next, we define curvature invariants related with $D$, see \cite{rov-112}, which for $D=TM$ are reduced to Chen's $\delta$-invariants, for example, \cite[Section~13.2]{chen1}.

\begin{definition}
\rm
Define Chen's type curvature invariants $\delta_D$ and $\hat\delta_D$ by
\begin{eqnarray}\label{E-ineq1-D}
\nonumber
 2\,\delta_D(n_1,\ldots, n_k)(x)=\tau_D(x)-\min\{\tau(V_1)+\ldots +\tau(V_k)\} ,\\
 2\,\hat\delta_D(n_1,\ldots, n_k)(x)=\tau_D(x)-\max\{\tau(V_1)+\ldots +\tau(V_k)\},
\end{eqnarray}
where $\tau(V_i)=\tr_g\Ric|_{V_i}$,
and $V_1,\ldots,V_k$ run over all $k\ge0$ mutually orthogonal subspaces of $D_x$ at $x\in M$ such that $\dim V_i=n_i\ (0\le i\le k)$.
\end{definition}

For example, $2\,\delta_D=\tau_D$ if $k=0$, and
$2\,\delta_D(n_1)(x)=\tau_D(x)-\min\tau(V_1)$
if $k=1$.

\begin{remark}\rm
In \eqref{E-ineq1-D}, we use max and min (instead of sup and inf, see \cite{chen1}) since the set ``all mutually orthogonal subspaces $V_1,\ldots,V_k$ at a point $x\in M$ such that ..." is~compact.
\end{remark}

Let $\{e_i\}$ be an orthonormal frame of a subspace $V=\bigoplus_{\,i=1}^{\,k} V_i$ of $T_x M$ such that
\[
 \{e_1,\ldots, e_{n_1}\}\subset V_1,\
 \ldots,\
 \{e_{n_{k-1}+1},\ldots, e_{n_k}\}\subset V_k.
\]
For $k\ge2$, the \textit{mutual curvature of a set} $\{V_1,\ldots,V_k\}$ is defined~by
\begin{equation}\label{E-mixed-k}
{\rm S}_{\rm m}(V_1,\ldots,V_k) = \sum\nolimits_{\,i<j} {\rm S}_{\rm m}(V_i, V_j),
\end{equation}
where
\[
 {\rm S}_{\rm m}(V_i, V_j) = \sum\limits_{\,n_{i-1}<a\,\le n_i,\ n_{j-1}<b\le n_j} g(R_{e_a,{e}_b}\,e_b,{e}_a)
\]
is the \textit{mutual curvature} of $(V_i, V_j)$.
Note that ${\rm S}_{\rm m}(V_1,\ldots,V_k)$ does not depend on the choice of frames.
We~get
\begin{equation}\label{E-Smix-3}
 \tau(V) = 2\,{\rm S}_{\rm m}(V_1,\ldots, V_k) +\sum\nolimits_{\,i=1}^k \tau(V_i),
\end{equation}
where $\tau(V)=\tr_g\Ric|_{\,V}$ is
the trace of the Ricci tensor on
$V=\bigoplus_{\,i=1}^{\,k} V_i$.
For example, if all subspaces $V_i$ are one-dimensional, then $2\,{\rm S}_{\rm m}(V_1,\ldots, V_k)=\tau(V)$.

We introduce the curvature invariants based on the concept of mutual curvature.

\begin{definition}\label{D-03}\rm
Define the \textit{mutual curvature invariants} of
a Riemannian manifold $(M^{d+l},g)$ equipped with a $d$-dimensional distribution $D$ by, see \cite{rov-112},
\begin{eqnarray}\label{E-delta-m}
\nonumber
 \delta^+_{{\rm m},D}(n_1,\ldots, n_k)(x) = \max {\rm S}_{\rm m}(V_1,\ldots,V_k),\\
 \delta^-_{{\rm m},D}(n_1,\ldots, n_k)(x) = \min {\rm S}_{\rm m}(V_1,\ldots,V_k),
\end{eqnarray}
where $x\in M$ and $V_1,\ldots,V_k$ run over all $k\ge2$ mutually orthogonal subspaces of $D_x$ such that $\dim V_i=n_i$ $(2\le i\le k)$.
For $D=TM$, we get the
invariants~$\delta^\pm_{{\rm m}}=\delta^\pm_{{\rm m},TM}$.
\end{definition}

The invariants in \eqref{E-ineq1-D} and \eqref{E-delta-m} are related by the following inequalities.

\begin{proposition}
For $k\ge2$ and $n_1+\ldots+n_k<d$, the following inequalities hold:
\begin{eqnarray*}
 \delta^+_{{\rm m},{D}}(n_1,\ldots, n_k) \ge \delta_{D}(n_1,\ldots, n_k) -\delta_{D}(n_1+\ldots+n_k)\,,\\
 \delta^-_{{\rm m},{D}}(n_1,\ldots, n_k) \le \hat\delta_{D}(n_1,\ldots, n_k) -\hat\delta_{D}(n_1+\ldots+n_k)\,,
\end{eqnarray*}
and if $n_1+\ldots+n_k=d$, then
\[
 \hat\delta_{D}(n_1,\ldots, n_k)=\delta^-_{{\rm m},D}(n_1,\ldots, n_k)\le
\delta^+_{{\rm m},D}(n_1,\ldots, n_k)=\delta_{D}(n_1,\ldots, \linebreak n_k).
\]
If the sectional curvature $K$ along $D$ satisfies $c\le K\le C$ and $\sum\nolimits_{\,i=1}^k n_i=s\le d$,~then
\begin{eqnarray}\label{E-s1}
\nonumber
  &\frac c2\,(s^2-\sum\nolimits_{\,i} n_i^2) = c\sum\nolimits_{\,i<j} n_i\,n_j \le \delta^-_{{\rm m},D}(n_1,\ldots n_k) \\
  &\le \delta^+_{{\rm m},D}(n_1,\ldots n_k) \le C\sum\nolimits_{\,i<j} n_i\,n_j =\frac C2(s^2-\sum\nolimits_{\,i} n_i^2)\,.
\end{eqnarray}
\end{proposition}

\begin{proof} This is similar to the proof of Proposition~1 in \cite{rov-112}.
\end{proof}

\begin{corollary}
If $(M^{d+l},D,g)$ has non-negative sectional curvature of planes tangent to $D$,
then
\[
 \hat\delta_{D}(n_1,\ldots, n_k) \le \delta^-_{{\rm m},D}(n_1,\ldots, n_k)\le
 \delta^+_{{\rm m},D}(n_1,\ldots, n_k)\le\delta_{D}(n_1,\ldots, n_k),
\]
and if this sectional curvature is nonpositive, then the above inequalities are opposite.
\end{corollary}

Given two $J$-invariant planes $\sigma$ and $\sigma'$ (2-dimensional subspaces) in $T_x M$ of an almost Hermitian manifold $(M,J,g)$, and unit vectors $X\in\sigma$ and $Y\in\sigma'$,
Goldberg and Kobayashi \cite{GK-1967} defined the \textit{holomorphic bisectional curvature} $K_{\rm h}(\sigma, \sigma')$ by
\begin{equation}\label{E-hol-bisect}
 K_{\rm h}(\sigma, \sigma') = R(X, JX, Y, JY) .
\end{equation}
This
depends on $\sigma$ and $\sigma'$ only, and for $\sigma=\sigma'$ gives the \textit{holomorphic sectional curvature}.
%
For a set of $J$-invariant planes in a complex distribution
of real dimension $d\ge 4$, we introduce invariants based on the holomorphic bisectional~curvature.

\begin{definition}\label{D-03H}\rm
Let $D$ be a $d$-dimensional complex distribution of a Riemannian manifold $(M, g)$, i.e., there is a skew-symmetric (1,1)-tensor $J:D\to D$ such that
\[
 J^{\,2} X=-X,\quad g(JX, JY)=g(X,Y),\quad X,Y\in D.
\]
The \textit{holomorphic mutual curvature invariants} $\delta^\pm_{{\,\rm h},D}(k)\ (1<k\le d/2)$ are defined by
\begin{equation*}
 \delta^+_{{\,\rm h},D}(k)(x) = \max {\rm S}_{\,\rm h}(\sigma_1,\ldots,\sigma_k) ,\quad
 \delta^-_{{\,\rm h},D}(k)(x) = \min {\rm S}_{\,\rm h}(\sigma_1,\ldots,\sigma_k) ,
\end{equation*}
where $\sigma_1,\ldots,\sigma_k$ run over all $k$ mutually orthogonal $J$-invariant planes of $D_x$ at a point $x\in M$,
and ${\rm S}_{\,\rm h}(\sigma_1,\ldots,\sigma_k)$ is defined using \eqref{E-hol-bisect} by
\begin{equation}\label{E-mixed-kH}
{\rm S}_{\,\rm h}(\sigma_1,\ldots,\sigma_k) = \sum\nolimits_{\,i<j} K_{\rm h}(\sigma_i, \sigma_j).
\end{equation}
For $D=TM$, i.e., for an almost Hermitian manifold,
we get the \textit{holomorphic mutual curvature invariants} $\delta^\pm_{\,\rm h}(k):=\delta^\pm_{\,{\,\rm h},TM}(k)$.
\end{definition}

\begin{lemma}
Let $\{\sigma_1,\ldots,\sigma_k\}\ (2\le k\le d/2)$ be mutually orthogonal $J$-invariant planes of a complex distribution $D_x$ at a point $x\in M$.
Then
\begin{equation}\label{E-mixed-bi}
 2\,{\rm S}_{\,\rm h}(\sigma_1,\ldots,\sigma_k) = {\rm S}_{\rm m}(\sigma_1,\ldots, \sigma_k) .
\end{equation}
\end{lemma}

\begin{proof}
By the Bianchi identity, we get
\begin{equation}\label{Eq-bisect}
 K_{\rm h}(\sigma, \sigma') = R(X, Y, X, Y) + R(X, JY, X, JY).
\end{equation}
Replacing $X$ with $JX$,
we get
 $K_{\rm h}(\sigma, \sigma') = R(JX, Y, JX, Y) + R(JX, JY, JX, JY)$.
Thus,
 ${\rm S}_{\rm m}(\sigma, \sigma') = 2\,K_{\rm h}(\sigma, \sigma')$.
From this, \eqref{E-mixed-k} and \eqref{E-mixed-kH}, we get \eqref{E-mixed-bi}.
\end{proof}

\begin{corollary}
The following inequalities are true for $2\le k\le d/2$:
\[
 2\,\delta^+_{{\,\rm h},D}(k) \le \delta^+_{{\rm m},D}(\underbrace{2,\ldots,2}_k),\quad
 2\,\delta^-_{{\,\rm h},D}(k) \ge \delta^-_{{\rm m},D}(\underbrace{2,\ldots,2}_k) .
\]
\end{corollary}

\section{CR-submanifolds in almost Hermitian manifolds}
\label{sec:03}

In this section, using mutual curvature invariants, Chen-type invariants and holomorphic mutual curvature invariants,
we prove several geometric inequalities for CR-submanifolds in almost Hermitian manifolds.

An even-dimensional Riemannian manifold $(\bar M, \bar g)$ equipped with a skew-sym\-metric {\rm (1,1)}-tensor $\bar J$ such that $\bar J^{\,2}X=- X$
and $\bar g(\bar J X,\bar J Y)=\bar g(X,Y)$ for all $X,Y\in T\bar M$ is called an {almost Hermitian manifold}.
We will put a top ``bar'' for objects related to $\bar M$.
%
A submanifold $M^{d+l}\ (d,l>0)$ of an almost Hermitian manifold $(\bar M, \bar J, \bar g)$ is called a \textit{CR-submanifold}
if $D=\bar J(TM)\cap TM$ is a complex distribution (the maxi\-mal $\bar J$-invariant subbundle) of constant real dimension~$d$, see \cite[Definition~7.2]{DO-2010}.
A different definition is given in \cite{b-1986,YK-1983}:
a real submanifold $M^{d+l}\ (d,l>0)$ of an almost Hermitian manifold $(\bar M, \bar J, \bar g)$ is called a CR-{submanifold} if there exists on $M$ a totally real distribution $D^\bot$
(i.e., $\bar J(D^\bot)\subset T^\bot M$) whose orthogonal complement $D$ (i.e., $TM = D\oplus D^\bot$) is a complex distribution (i.e., $\bar J(D)=D$) of constant real dimension $d$.
Both definitions above give the same thing when
$\dim D$ is maximum, that is,
$D^\perp$ is one-dimensional. The~main examples are real hypersurfaces, others examples are explained in
\cite{DO-2010}.


\begin{example}\rm
Let $(M^{d+1},g)$ be a CR-submanifold of an almost Hermitian manifold $(\bar M,\bar J,\bar g)$
with a $d$-dimensional complex distribution $D=\bar J(TM)\cap TM$.
Then $M^{d+1}$ admits an almost contact metric structure $(\varphi,\xi,\eta,g)$,
where $\varphi{=}\bar J|_{D}$ and $\xi$ is the unit tangent vector field orthogonal to~$D$.
\end{example}\rm

Let ${h}:TM\times TM\to TM^\bot$ be the 2nd fundamental form of the submanifold $(M,g)$ of the Riemannian manifold $(\bar M, \bar g)$.
 Recall the Gauss equation,
 see~\cite{chen1}:
\begin{equation}\label{E-Gauss-class}
  \bar g(\bar R_{Y,Z}\,U,X) = g(R_{Y,Z}\,U,X) + g({h}(Y,U), {h}(Z,X)) - g({h}(Z,U), {h}(Y,X)) ,
\end{equation}
where $U,X,Y,Z\in TM$ and $\bar R$ and $R$ are the curvature tensors of $(\bar M,\bar g)$ and $(M,g)$, respectively.
The mean curvature vector field of a subspace $V\subset T_x M$ is given by
$H_V = \sum_i h(e_i,e_i)$, where $e_i$ is an orthonormal basis of~$V$.
In short form, we will write $H_i$ instead of $H_{V_i}$,
and $H$ if $V=T_x M$.
For a CR-submanifold $(M^{d+l},g)$,~set
\begin{equation*}
 {\cal H}_{D_x}(s) = \max \{\,\|{H}_V\,\|: \ V\subset D_x, \ \dim V=s>0\}.
\end{equation*}
If $s=d$, then ${\cal H}_{D_x}(d)=\|H_{D_x}\|$, where $H_{D_x}$ is the mean curvature vector of~$D_x$.
Note that for $s<d$, the equality ${\cal H}_{D_x}(s)=0$ implies $h\,|_{D_x}=0$.

A CR-submanifold $(M,g)$
in an almost Hermitian space $(\bar M,\bar J,\bar g)$ is called ${D}$-\textit{minimal}
(where $D=\bar J(TM)\cap TM$)
if ${H}_{D}\equiv0$.
A CR-submanifold $(M^{d+l},g)$ is called \textit{mixed totally geodesic} on $V=\bigoplus_{\,i=1}^{\,k} V_i\subset D$ if
 ${h}(X,Y)=0$
 for all
 $X\in{V}_i,\,Y\in{V}_j$ and $i\ne j$.
Note~that
\[
 \delta^+_{{\rm m},{D}\,}(n_1,\ldots,n_k)\le \delta^+_{{\rm m}}(n_1,\ldots,n_k),
\]
and, for $s < d$, by
${\cal H}_{{D}}(s)=0$, $M$ is totally geodesic on ${D}$, i.e., $h\,|_{{D}_x}=0$.

\begin{theorem}\label{T-V}
Let~$(M^{d+l},g)$ be a CR-submanifold of an almost Hermitian mani\-fold $(\bar M,\bar J,\bar g)$, and $D=\bar J(TM)\cap TM$.
For any natural numbers $n_1,\ldots,n_k$ such that $\sum_{\,i}n_i=s\le d$, we~obtain
\begin{equation}\label{E-ineq-V1}
 \delta^+_{{\rm m},D}(n_1,\ldots,n_k) \le \bar\delta^+_{\rm m}(n_1,\ldots,n_k) + \frac{k-1}{2\,k}
 \left\{\begin{array}{cc}
  {\cal H}_{D}(s)^2,  & {\rm if}\ s<d, \\
  \|\,{H}_{D}\,\|^2, & {\rm if}\ s=d,
\end{array}\right.
\end{equation}
where $\bar\delta^+_{\rm m}(n_1,\ldots,n_k)$ are defined for $(\bar M, \bar g)$ similarly to $\delta^+_{\rm m}(n_1,\ldots,n_k)$ for $(M,g)$,
see Definition~\ref{D-03}.
The equality in \eqref{E-ineq-V1} holds at a point $x\in M^{d+l}$ if and only if
there exist mutually orthogonal subspaces ${V}_1, \ldots, {V}_k$ of $D_x$ with $\sum_{\,i}n_i=s$ such that
$M^{d+l}$ is mixed totally geodesic on $V=\bigoplus_{\,i=1}^{\,k} V_i$, ${H}_1=\ldots={H}_k$, \ $\|H_V\|={\cal H}_{D_x}(s)$
and $\bar{\rm S}_{\rm m}({V}_1, \ldots, {V}_k)=\bar\delta^+_{\rm m}(n_1,\ldots,n_k)(x)$.
\end{theorem}

\begin{proof}
Taking a trace of
\eqref{E-Gauss-class} for the submanifold $M^{d+l}$ on $V$ and $V_i$ yields
\begin{eqnarray}\label{E-Si}
  \bar\tau(V) - \tau(V) = \|{h}_{V}\|^2 - \|{H}_{V}\|^2, \quad
  \bar\tau({V}_i) - \tau({V}_i) = \|{h}_i\|^2 - \|{H}_i\|^2,
\end{eqnarray}
where $\bar\tau(V)$, $\bar\tau({V}_i)$ and $\tau(V)$, $\tau({V}_i)$ are the scalar curvatures of subspaces $V=\bigoplus_{\,i=1}^{\,k} V_i$ and ${V}_i$ for the curvature tensors $\bar R$ and $R$, respectively,
at the point $x\in M$.

Let ${H}_V\ne0$ hold on an open set $U\subset M$.
We complement an adapted local orthonormal frame $\{e_1,\ldots,e_d\}$ of $D$ over $U$ with a vector field $e_{d+1}$ parallel to~${H}_V$.
By $H_V{=}\sum_{\,i=1}^{\,k} H_i$ and
$a_1^2{+}\ldots{+}a_k^2\ge \frac1k\,(a_1{+}\ldots{+}a_k)^2$
for $a_i=\bar g({H}_i, e_{d+1})$, we~get
\begin{equation}\label{E-Si3}
  \sum\nolimits_{\,i}\|{H}_i\|^2 \ge \sum\nolimits_{\,i} \bar g({H}_i, e_{d+1})^2 \ge \frac1k\,\|{H}_V\|^2,
\end{equation}
and the equality holds if and only if ${H}_1=\ldots={H}_k$.
The
\eqref{E-Si3} is
true for ${H}_V=0$; thus, it is valid on $M$.
Set $\|{h}^{\rm mix}_{ij}\|^2=\sum_{\,e_a\in{V}_i,\, e_b\in{V}_j} \|{h}(e_a,e_b)\|^2$ for $i\ne j$.
Note that
\begin{equation}\label{E-Si4}
 \|{h}_V\|^2 = \sum\nolimits_{\,i}\|{h}_i\|^2 +\sum\nolimits_{\,i<j}\|{h}^{\rm mix}_{ij}\|^2 \ge \sum\nolimits_{\,i}\|{h}_i\|^2 ,
\end{equation}
and the equality holds if and only if $\|{h}^{\rm mix}_{ij}\|^2=0\ (\forall\,i<j)$, i.e., $M^{d+l}$ is mixed totally geodesic along $V$.
By \eqref{E-Si}--\eqref{E-Si4} and the following equalities, see \eqref{E-Smix-3}:
\begin{equation*}
 \bar\tau(V) = 2\,\bar{\rm S}_{\rm m}({V}_1,\ldots,{V}_k) +\sum\nolimits_{\,i}\bar\tau({V}_i),\quad
 \tau(V) = 2\,{\rm S}_{\rm m}({V}_1,\ldots,{V}_k) +\sum\nolimits_{\,i}\tau({V}_i),
\end{equation*}
we obtain
\begin{eqnarray*}
\nonumber
 && 2\,{\rm S}_{\rm m}({V}_1,\ldots,{V}_k) =  2\,\bar{\rm S}_{\rm m}({V}_1,\ldots,{V}_k)
 + \sum\nolimits_{\,i} \big(\bar\tau({V}_i) - \tau({V}_i)\big) + \|{H}_V\|^2 - \|{h}_V\|^2 \\
\nonumber
 && \le 2\,\bar\delta^+_{{\rm m},D}(n_1,\ldots,n_k) - ( \|{h}_V\|^2 - \sum\nolimits_{\,i}\|{h}_i\|^2 )
  + (\|{H}_V\|^2 - \sum\nolimits_{\,i}\|{H}_i\|^2 ) \\
 && \le 2\,\bar\delta^+_{{\rm m},D}(n_1,\ldots,n_k) + \frac{k-1}k\,{\cal H}_D(s)^2 ,
\end{eqnarray*}
and the equality holds in the 2nd line if and only if
$\bar{\rm S}_{\rm m}({V}_1,\ldots,{V}_k){=}\bar\delta^+_{{\rm m},D}(n_1,\ldots,n_k)$
and $\|H_V\|={\cal H}_x(s)$ at each point $x\in M$. This proves \eqref{E-ineq-V1} for $s<d$.
The case $\sum_{\,i}n_i=d$ of \eqref{E-ineq-V1} can be proved similarly.
\end{proof}

\begin{remark}\rm
For a CR-submanifold $(M^{d+l},g)$ in an almost Hermitian space $(\bar M,\bar J,\bar g)$ with sectional curvature bounded above by $c$,
for $\sum_{\,i}n_i=s\le d$ from \eqref{E-s1} and \eqref{E-ineq-V1},~we~get
\begin{equation}\label{E-ineq-c}
 \delta^+_{{\rm m},D}(n_1,\ldots,n_k) \le
  \left\{\begin{array}{cc}
  \frac c2\,(s^2-\sum\nolimits_{\,i} n_i^2) + \frac{k-1}{2\,k}\,{\cal H}_{D}(s)^2  ,  & {\rm if}\ s<d, \\
  \frac c2\,(d^2-\sum\nolimits_{\,i} n_i^2) + \frac{k-1}{2\,k}\,\|{H}_{D}\,\|^2, & {\rm if}\ s=d.
\end{array}\right.
\end{equation}
For $s=d$, the RHS of \eqref{E-ineq-c} coincides with the RHS of
(13.43) in \cite{chen1} for $\sum_{\,i}n_i=d$.
\end{remark}

As real hypersurfaces of almost Hermitian manifolds are the main examples of CR-submani\-folds and important
objects in the study of geometrical inequalities, we reformulate Theorem~\ref{T-V} especially for this case.

\begin{corollary}
Let~$(M^{2n+1}, g)$ be a real hypersurface with a complex distribution $D=\bar J(TM)\cap TM$ of an almost Hermitian manifold $(\bar M,\bar J,\bar g)$.
For any natural numbers $n_1,\ldots,n_k$ such that $\sum_{\,i}n_i=s\le 2n$, we obtain the inequality
\begin{equation}\label{E-ineq-V3}
 \delta^+_{{\rm m},D}(n_1,\ldots,n_k) \le \bar\delta^+_{\rm m}(n_1,\ldots,n_k) + \frac{k-1}{2\,k}
 \left\{\begin{array}{cc}
  {\cal H}_{D}(s)^2,  & {\rm if}\ s<2n, \\
  \|\,{H}_{D}\,\|^2, & {\rm if}\ s=2n.
\end{array}\right.
\end{equation}
The equality in \eqref{E-ineq-V3} holds at a point $x\in M^{2n+1}$ if and only if
there exist mutually orthogonal subspaces ${V}_1, \ldots, {V}_k$ of $D_x$ with $\sum_{\,i}n_i=s$ such that
$M^{2n+1}$ is mixed totally geodesic on $V=\bigoplus_{\,i=1}^{\,k} V_i$, ${H}_1=\ldots={H}_k$, \ $\|H_V\|={\cal H}_{D_x}(s)$
and $\bar{\rm S}_{\rm m}({V}_1, \ldots, {V}_k)=\bar\delta^+_{\rm m}(n_1,\ldots,n_k)(x)$.
\end{corollary}

For any $k$-tuple $(n_1, \ldots , n_k)$ with $\sum_{\,i}n_i=s\le d$, define the \textit{normalized $\delta_{{\rm m},D}$-curvature}~by
\[
 \Delta_{{\rm m},D}(n_1, \ldots , n_k)=\frac{2\,k}{k-1}\,\delta^+_{{\rm m},D}(n_1, \ldots , n_k),
\]
and for $\sum_{\,i}n_i=s$ put $\bar\Delta_{{\rm m},D} := \max \Delta_{{\rm m},D}(n_1, \ldots , n_k)$.

\noindent\ \
Theorem~\ref{T-V} gives the following (compare with the maximum principle \cite[p.~268]{chen1}).

\begin{proposition}

a CR-sub\-manifold $(M^{d+l},g)$ of $\mathbb{C}^{q}$
for some $k$-tuple $(n_1, \ldots, n_k)$ with $\sum_{\,i}n_i=s\le d$,
then for all $(m_1, \ldots, m_k)$ with $\sum_{\,i}m_i=s$,
we~get
 $\Delta_{{\rm m},D}(n_1, \ldots , n_k)\ge\Delta_{{\rm m},D}(m_1, \ldots, m_k)$.
\end{proposition}

\begin{proof}
By the conditions, $\Delta_{{\rm m},D}(n_1, \ldots, n_k) {=} \bar\Delta_{{\rm m},D}(s)$.
Since $\Delta_{{\rm m},D}(m_1, \ldots, m_k)\le {\cal H}_{D}(s)^2$, we obtain the inequality
$\Delta_{{\rm m},D}(m_1, \ldots, m_k)\le \Delta_{{\rm m},D}(n_1, \ldots , n_k)$.
\end{proof}

\begin{corollary}\label{C-03}
For every CR-submanifold $(M^{d+l},g)$ in $\mathbb{C}^{q}$, we have
${\cal H}_{D}(s)^2 \ge \bar\Delta_{{\rm m},D}(s)$ for any $s<d$, and $\|{H}_{D}\,\|^2 \ge \bar\Delta_{{\rm m},D}$.
\end{corollary}

The case of equality in Corollary~\ref{C-03} is of special interest.
Such extremal CR-immersions in $\mathbb{C}^{q}$ can be compared to ``ideal immersions"
introduced by Chen's for real space forms in terms of $\delta$-invariants, for example, \cite[Definition~13.3]{chen1}.

The theory of $\delta_{D}$-invariants \eqref{E-ineq1-D} of CR-submanifolds can be developed similarly to the theory of Chen's $\delta$-invariants of a Riemannian submanifold.

\begin{theorem}
Let~$(M^{d+l}, g)$ be a CR-submanifold of an almost Hermitian manifold $(\bar M,\bar J,\bar g)$ with sectional curvature bounded above by $c\in\RR$.
For each $k$-tuple $(n_1,\ldots,n_k)$
such that $\sum_{\,i}n_i\le d$, we obtain (similarly to \cite[Theorem~13.5]{chen1})
\begin{equation}\label{E-ineq-new}
 \delta_{D}(n_1,\ldots,n_k)\le \frac{d^2(d+k-1-\sum_{\,i} n_i)}{2(d+k-\sum_{\,i} n_i)} \|H_{D}\|^2
+ \frac c2\,[d(d-1)- \sum\nolimits_{\,i} n_i(n_i-1)].
\end{equation}
\end{theorem}

The case of equality in \eqref{E-ineq-new} is of special interest: extremal CR-submanifolds
in terms of $\delta_{D}$-invariants
are an analogue of Chen's ``ideal immersions".

\smallskip

Set
 $\delta^+_{{\rm m},D}(k){=}\max \delta^+_{{\rm m},D}(n_1,\ldots, n_k)$
 and
 $\delta^-_{{\rm m},D}(k){=}\min \delta^-_{{\rm m},D}(n_1,\ldots, n_k)$,
where $\sum_{\,i}n_i\le d$.
The $\bar\delta^+_{{\rm m}}(k+1)$ are defined for $(\bar M, \bar g)$ similarly to $\delta^+_{{\rm m}}(k+1)$ for $(M,g)$.


\begin{theorem}\label{T-V2}
Let~$(M^{d+l}, g)$ be a CR-submanifold of an almost Hermitian mani\-fold $(\bar M,\bar J,\bar g)$.
For any $k\ge2$, we~obtain the inequality that supplements~\eqref{E-ineq-V3}:
\begin{eqnarray}
\label{E-ineq-V2}
 && \delta^-_{{\rm m},D}(k) \le \frac{k-1}{2\,k(k+1)}\,\|{H}_{D}\,\|^2 +\bar\delta^+_{{\rm m}}(k+1)\,.
\end{eqnarray}
The equality in \eqref{E-ineq-V2} holds at a point $x\in M^{d+l}$ if and only if
there exist mutually orthogonal subspaces ${V}_1, \ldots, {V}_{k+1}$ of $D_x$ with $\sum_{\,i=1}^{k+1}n_i = d$ such that
$M^{d+l}$ is mixed totally geodesic, ${H}_1=\ldots={H}_{k+1}$,
$\bar{\rm S}_{\rm m}({V}_1, \ldots, {V}_{k+1})=\bar\delta^+_{{\rm m}}(n_1,\ldots,n_{k+1})$ and
${\rm S}_{\rm m}(V_1,\ldots,\widehat V_{i},\ldots, V_{k+1}) = \delta^-_{{\rm m},D}(k)$
for any $i=1,\ldots,k+1$, where $\widehat V_{i}$ means removing the space $V_i$ from the set $\{V_1,\ldots, V_{k+1}\}$.
\end{theorem}

\begin{proof} Let $V_{k+1}$ be the orthogonal complement to $V=\bigoplus_{\,i=1}^{\,k} V_i$ in $D_x$. Note that
\[
 \sum\nolimits_{\,i}{\rm S}_{\rm m}(V_1,\ldots,\widehat V_{i},\ldots, V_{k+1})= (k+1)\,{\rm S}_{\rm m}(V_1,\ldots, V_{k+1}).
\]
We also obtain
 $\delta^-_{{\rm m},D}(k)\le \delta^-_{{\rm m},D}(n_1,\ldots,\widehat n_{i},\ldots, n_{k+1})\le {\rm S}_{{\rm m},D}(V_1,\ldots,\widehat V_{i},\ldots, V_{k+1})$
for any $i=1,\ldots,k+1$.
Thus, $\delta^-_{{\rm m},D}(k)\le {\rm S}_{\rm m}(V_1,\ldots, V_{k+1})$, and using \eqref{E-ineq-V3}
for $\sum_{\,i}n_i=d$ gives~\eqref{E-ineq-V2}.
\end{proof}

\noindent\ \
Theorems~\ref{T-V} and \ref{T-V2} give the assertions on the absence of some CR-submanifolds.

\begin{corollary}
There are no $D$-minimal CR-submanifolds $(M^{d+l},g)$ in $\mathbb{C}^{q}$ with any of the following pro\-perties:

{\rm (a)}~
$\delta^+_{{\rm m},D}(n_1,\ldots,n_k)>0$ for some $(n_1,\ldots,n_k)$ with $\sum_{\,i}n_i = d$,

{\rm (b)}~
$\delta^-_{{\rm m},D}(k)>0$ for some $k\ge2$.
\end{corollary}


Next, we use the fact that the tangent distribution $TM$ is the sum $TM={D}\oplus{D}^\bot$ of two mutually orthogonal distributions
$D$ and $D^\bot$ of ranks $d$ and $l$.
 Let $x\in M$ and $\{e_i\}$ on $(M,g)$ be an adapted orthonormal frame, i.e.,
 $\{e_1,\ldots, e_{d}\}\subset{{D}(x)},
 \
 \{e_{d+1},\ldots, e_{d+l}\}\subset{{D}^\bot(x)}$.
The~\textit{mutual curvature} of $({D},{D}^\bot)$ is a function ${\rm S}_{\rm m}({D},{D}^\bot)$,
 given at $x\in M$ by
 ${\rm S}_{\rm m}({D}(x),{D}^\bot(x)) {=} \sum\nolimits_{\,1\le a\,\le d,\ d<b\le d+l} K(e_a,{e}_b)$.
In this case, ${\rm S}_{\rm m}({D},{D}^\bot)$ is the \textit{mixed scalar curvature}; see~\cite{r-w-2022}.
A CR-submanifold $(M,g)$ is called \textit{mixed totally geodesic} on $({D},D^\bot)$ if ${h}(X,Y)=0\ (X\in{D},\ Y\in{D}^\bot)$.

\begin{theorem}
Let $(M^{d+l},g; {D},{D}^\bot)$, where $D=\bar J(TM)\cap TM$, be a CR-submani\-fold of an almost Hermitian manifold $(\bar M,\bar J,\bar g)$.
Then the following inequality holds:
\begin{equation}\label{E-ineq-k}
 {\rm S}_{\rm m}({D},{D}^\bot) \le ({1}/{4})\,\|H_D\|^2 +\bar\delta^+_{\rm m}(d,l)\,.
\end{equation}
The equality in \eqref{E-ineq-k} holds at a point $x\in M$ if and only if $M^{d+l}$ is mixed totally geodesic,
${H}_D(x)={H}_{D^\bot}(x)$ and $\bar{\rm S}_{\rm m}({D}(x), {D}^\bot(x))=\bar\delta^+_{\rm m}(d,l)(x)$.
\end{theorem}

\begin{proof} The proof of the first statement is similar to the proof of Theorem~\ref{T-V}.
The second assertion follows directly from the cases of equality, as for Theorem~\ref{T-V}.
\end{proof}

\begin{corollary}
\!\!A CR-submanifold
in $\mathbb{C}^{q}$ with ${\rm S}_{\rm m}({D},{D}^\bot){>}0$ cannot be ${D}$-minimal.
\end{corollary}

\begin{example}\rm
Consider distributions ${D},D^\bot$ on a domain $M$ on a unit sphere $S^{d+l}(1)$ in a complex Euclidean space;
thus, $\bar\delta^+_{\rm m}(d,l)=0$.
Using coordinate charts, we can take integrable distributions ${D}, D^\bot$,
and $M$ is diffeomorphic to the product of two manifolds.

Let $d=2$ and $l=1$; then, $\|{H}\,\|^2=9$ and ${\rm S}_{\rm m}({D},{D}^\bot) = 2$.
Hence, \eqref{E-ineq-k} reduces to the inequality $2 < 9/4$. Note that ${H}_D=\frac13\,{H} \not = \frac23\,{H} ={H}_{D^\bot}$.

Let $d=l=2$ and locally $M\subset S^4(1)$ be diffeomorphic to $\mathbb{C} \times \mathbb{C}$.
Then, $\|{H}\,\|^2=16$, ${H}_D={H}_{D^\bot}$,
${\rm S}_{\rm m}({D},{D}^\bot) = 4$ and \eqref{E-ineq-k} reduces to the equality $4 = 16/4$.
\end{example}

The following theorem deals
with holomorphic bisectional curvature invariants.

\begin{theorem}\label{T-kH}
Let~$(M^{d+l}, g)$ be a CR-submanifold of an almost Hermitian mani\-fold $(\bar M,\bar J,\bar g)$.
For any natural number $k\in[2, d/2]$, we~obtain
\begin{equation}\label{E-ineq-V1H}
 \delta^+_{\,{\,\rm h},D}(k) \le \bar\delta^+_{\,\rm h}(k) + \frac{k-1}{4\,k}
 \left\{\begin{array}{cc}
  {\cal H}_{D}(2k)^2,  & {\rm if}\ 2k<d, \\
  \|\,{H}_{D}\,\|^2, & {\rm if}\ 2k=d,
\end{array}\right.
\end{equation}
where $\bar\delta^+_{\,\rm h}(k)$ are defined for $(\bar M, \bar g)$ similarly to $\delta^+_{\,\rm h}(k)$ for $(M,g)$,
see Definition~\ref{D-03H}.
The equality in \eqref{E-ineq-V1H} holds at $x\in M^{d+l}$ if and only if
there exist mutually orthogonal $J$-invariant planes $\{\sigma_1, \ldots, \sigma_k\}$ of $D_x$ such that
$M^{d+l}$ is mixed totally geodesic on $V{=}\bigoplus_{\,i=1}^{\,k} \sigma_i$, ${H}_1=\ldots={H}_k$, \ $\|H_V\|={\cal H}_{D_x}(2k)$
and $\bar{\rm S}_{\,\rm h}(\sigma_1, \ldots, \sigma_k)=\bar\delta^+_{\,\rm h}(k)(x)$.
\end{theorem}

\begin{proof}
 This is similar to the proof of Theorem~\ref{T-V}.
\end{proof}

\begin{corollary}\label{C-new-1H}
Let~$(M^{d+1}, g)$ be a real hypersurface of an almost Hermitian mani\-fold $(\bar M,\bar J,\bar g)$.
Then \eqref{E-ineq-V1H} is true for any natural number $k\in[2, d/2]$.
\end{corollary}

Using $\delta(2,\ldots,2)$-invariants, Chen classified in \cite[Section~15.7]{chen1} extremal real hypersurfaces of K\"{a}hlerian space~forms.
Similarly, we would like to study the extreme case of Corollary~\ref{C-new-1H} when $(\bar M,\bar J,\bar g)$ is a K\"{a}hlerian space~form.

From Theorem~\ref{T-kH} we get the assertion on the absence of some CR-submanifolds.

\begin{corollary}
A CR-submanifold $(M^{d+l}, g)$ in $\mathbb{C}^{q}$ satisfying  $\delta^+_{{\,\rm h},D}(d/2)>0$ cannot be $D$-minimal.
\end{corollary}

\section{Conclusions}

We studied the question {of finding a simple optimal connection between the intrinsic and extrinsic invariants of
a manifold equipped with a complex distribution.
The~main contribution of the paper is the concept of curvature invariants $\delta^\pm_{{\rm m},D}$ of
CR-submani\-folds of almost Hermitian manifolds}, based on the mutual curvature of several pairwise orthogonal subspaces of a contact distribution~$D$.
We used these curvature invariants and Chen-type curvature invariants $\delta^\pm_{D}$
to prove new geometric inequalities invol\-ving the squared intermediate mean curvature for CR-submanifolds of almost Hermitian manifolds.
In~the case of
complex planes, we study curvature invariants $\delta^\pm_{\,{\,\rm h},D}$ based on the concept of holomorphic bisectional curvature.
Consequences of the absence of some $D$-minimal CR-submanifolds were~provided.



\end{document}